\documentclass{amsart}
\usepackage{amsmath}
\usepackage{amsthm}
\usepackage{amssymb}

\numberwithin{equation}{section}
\newtheorem{theorem}[equation]{Theorem}
\newtheorem{maintheorem}{Theorem}

\newtheorem{lemma}[equation]{Lemma}
\newtheorem{proposition}[equation]{Proposition}
\newtheorem{corollary}[equation]{Corollary}

\newtheorem{question}[equation]{Question}
\theoremstyle{definition}
\newtheorem{definition}[equation]{Definition}

\newtheorem{notation}[equation]{Notation}
\newtheorem{remark}[equation]{Remark}
\DeclareMathOperator{\Spec}{Spec}
\DeclareMathOperator{\Ext}{Ext}
\DeclareMathOperator{\Tor}{Tor}
\DeclareMathOperator{\Hom}{Hom}
\DeclareMathOperator{\RHom}{RHom}

\DeclareMathOperator{\ob}{ob}
\newcommand{\cat}[1]{\mathsf{#1}}
 \newcommand{\Mod}{\cat{Mod}}
\newcommand{\thick}{\cat{thick}}

\newcommand{\derived}[1]{\cat{D}(#1)}
\newcommand{\latt}{\mathbf{m}}
\newcommand{\Latt}{\mathbf{M}}
\newcommand{\B}{\cat{B}}
\newcommand{\DL}{\cat{DL}}
\newcommand{\BA}{\cat{BA}}
\newcommand{\bousfield}[1]{\left\langle #1 \right\rangle}
\newcommand{\Smash}{\ltensor{\Alg}}
\newcommand{\mycomp}{^{c}}
\newcommand{\I}[1]{I(#1)}
\newcommand{\M}[1]{\Alg(#1)}
\newcommand{\alg}[1]{\Alg(#1)}
\newcommand{\Alg}{\Lambda}
\renewcommand{\k}[1]{k(#1)}
\newcommand{\module}[1]{M_{#1}}
\newcommand{\restrict}[2]{#1\rvert_{#2}}
\newcommand{\N}{\mathbf{N}}
\newcommand{\tensor}[1]{\underset{#1}{\otimes}}
\newcommand{\ltensor}[1]{\underset{#1}{\overset{L}{\otimes}}}
\newcommand{\incl}[2]{\iota_{#1}^{#2}}
\newcommand{\bcdual}[1]{I ({#1})}

\newcommand{\comm}{\sim}
\newcommand{\lecomm}{\lesssim}
\newcommand{\suchthat}{\,:\,}
\begin{document}

\title[Truncated polynomial algebras]
{The Bousfield lattice for \\ truncated polynomial algebras}

\author{W. G. Dwyer}
\address{(Dwyer) Department of Mathematics, University of Notre Dame,
Notre Dame, IN 46556, USA}
\email{dwyer.1@nd.edu}

\author{J. H. Palmieri}
\address{(Palmieri) Department of Mathematics, Box 354350, University
of Washington, Seattle, WA 98195, USA}
\email{palmieri@math.washington.edu}

\subjclass{18E30,  
  55U35, 
  13D07, 
  13D25 
  }

\keywords{Bousfield lattice, derived category, commutative ring}

\begin{abstract}
The global structure of the unbounded derived category of a truncated
polynomial ring on countably many generators is investigated, via its
Bousfield lattice.  The Bousfield lattice is shown to have cardinality
larger than that of the real numbers, and objects with large
tensor-nilpotence height are constructed.
\end{abstract}

\maketitle

\section{Introduction}

Suppose that $R$ is a commutative ring.  In $\derived{R}$, the
unbounded derived category of $R$, the derived tensor product
$-\ltensor{R}-$ is exact, associative, and commutative.  Inspired by
stable homotopy theory, we make the following definition: given an
object $X$ in $\derived{R}$, its \emph{Bousfield class} is
\[
\bousfield{X} = \{W \in \ob \derived{R} \suchthat X \ltensor{R} W = 0 \}.
\]
We say that objects $X$ and $Y$ are \emph{Bousfield equivalent} if
$\bousfield{X}=\bousfield{Y}$.  The Bousfield classes form a lattice
$\B=\B (R)$, where the partial ordering is by reverse
inclusion.\footnote{We note that it is not clear in general that the
Bousfield classes form a set.  If $R$ is countable or noetherian, then
they do; otherwise, it is unknown.  See Question~\ref{question-set}.}.
For example, $\bousfield{0}$ is the smallest element and
$\bousfield{R}$ is the largest one.  The study of stable homotopy
theory suggests that the structure of this lattice contains useful
information about the category $\derived{R}$.

For example, a theorem of Neeman \cite{neeman;derived} says that if
$R$ is commutative and noetherian, then the Bousfield lattice of
$\derived{R}$ is isomorphic to the Boolean algebra of subsets of
$\Spec (R)$.  The purpose of this paper is to investigate the
situation for a non-noetherian ring.

Fix a field $k$ and integers $n_{i} \geq 2$ for all $i \geq 1$.  For
any integer $m \geq 1$, consider the truncated polynomial algebra
\[
\Alg_{m} = k[x_{1}, x_{2}, \dotsc, x_{m}] / (x_{i}^{n_{i}} 
\ \text{for} \ 1 \leq i \leq m).
\]
Grade $\Alg_{m}$ so that it is locally finite and graded connected, and
also so that it is graded commutative: for example, set $\deg x_{i} =
2^{i}$.  Since $\Spec (\Alg_{m})$ contains only one element, the ideal
consisting of all elements of positive degree, Neeman's theorem says
that the Bousfield lattice of $\derived{\Alg_{m}}$ is not very
interesting: it has two elements, $\bousfield{0}$ and
$\bousfield{\Alg_{m}}$, and any two nonzero objects of $\derived{\Alg_{m}}$
are Bousfield equivalent to each other.

This situation gets much more complicated if we remove the noetherian
condition.  Let
\begin{equation}\label{eqn-trunc}
\Alg = k[x_{1}, x_{2}, x_{3}, \dotsc] / (x_{i}^{n_{i}} \ \text{for all} \ i)
\end{equation}
be a truncated polynomial algebra on countably many generators, graded
as $\Alg_{m}$ is.  Assume that the base field $k$ is countable, so
that $\Alg$ is countable and its Bousfield lattice forms a set.  In
this paper, we investigate this Bousfield lattice.  For example, we
have the following theorem.

\begin{maintheorem}\label{thm-uncountable}
The Bousfield lattice $\B = \B (\Alg)$ of $\derived{\Alg}$ has
cardinality at least $2^{2^{\aleph_{0}}}$.
\end{maintheorem}

This is proved as Theorem~\ref{thm-big}; see also the discussion
at the end of Section~\ref{sec-more-bousfield}.  The Bousfield lattice
has a largest element $\bousfield{\Alg}$, a smallest element
$\bousfield{0}$, and a unique minimum nonzero element
$\bousfield{\I{\N}}$.  We construct a sublattice $\latt$ so that given
$\bousfield{\I{\N}} \leq \bousfield{X} < \bousfield{Y} \leq
\bousfield{\Alg}$ in $\latt$, there is an uncountable antichain in the
interval between $\bousfield{X}$ and $\bousfield{Y}$ in $\latt$.  The
various joins of these elements form a subset of $\B$ of cardinality
$2^{2^{\aleph_{0}}}$.

The minimality of the class $\bousfield{\I{\N}}$ has some interesting
consequences, including a proof of the telescope conjecture for the
category $\derived{\Alg}$ -- see Section~\ref{sec-minimal}.

The derived category of this ring has other interesting features; for
example, we have the following, which is proved below as
Theorem~\ref{thm-tensor}.

\begin{maintheorem}\label{thm-tensor-height}
For any integer $n \geq 1$, there is an object $X_{n}$ in
$\derived{\Alg}$ so that the $n$-fold derived tensor product of
$X_{n}$ with itself is nonzero, while the $(n+1)$-fold derived tensor
product is zero.
\end{maintheorem}

In fact, the object $X_{n}$ may be taken to be a $\Alg$-module.

We point out that the operations $\oplus$ and $\ltensor{R}$ on
$\derived{R}$ induce operations on Bousfield classes, and it is easy
to check that $\bousfield{X} \oplus \bousfield{Y}$ is the least upper
bound, or join, of $\bousfield{X}$ and $\bousfield{Y}$; this holds
more generally for a direct sum indexed by any set.  The tensor
product operation is more complicated.  In the noetherian case, one
can use Neeman's result to show that $\bousfield{X} \ltensor{R}
\bousfield{Y}$ is the greatest lower bound, or meet, of
$\bousfield{X}$ and $\bousfield{Y}$, but
Theorem~\ref{thm-tensor-height} shows that this need not be true in the
non-noetherian case.

We also have a nilpotence theorem (Theorem~\ref{thm-nilp}), a thick
subcategory theorem (Corollary~\ref{cor-thick}), and a periodicity
theorem (Proposition~\ref{prop-periodicity}), \`a la Hopkins and Smith
\cite{hopkins-smith, hopkins;global}.

\subsection{Philosophy}

A key idea behind this work is that the derived category $\derived{R}$
of a commutative ring $R$ is a stable homotopy category in the sense
of \cite{axiomatic}, and so shares formal properties with the stable
homotopy category of spectra.  One can take this in two different
directions: one can take results (or questions) from stable homotopy
theory and apply them (or try to answer them) in $\derived{R}$.  This
can produce results which are interesting on their own, in that they
shed some light on the structure of the category $\derived{R}$ and
hence into the ring $R$ and its modules.
Theorem~\ref{thm-tensor-height} is an example, as is the nilpotence
theorem ~\ref{thm-nilp}.

One can also try to use derived categories as test cases for questions
and conjectures about the stable homotopy category of spectra.  Neeman
\cite{neeman;derived} has explored the derived category of a
commutative noetherian ring $R$; his work gives us a pretty clear
picture of what $\derived{R}$ looks like through the lens of stable
homotopy theory.  One might summarize Neeman's results by saying that
the prime ideal spectrum of $R$ governs much of the global structure
of the derived category.  More generally, if one has any stable
homotopy category in which ``the homotopy groups of the sphere'' forms
a noetherian ring $R$, one can investigate the extent to which
Neeman's results still hold, which is to say, how similar the given
stable homotopy category is to $\derived{R}$; this is done, in part,
in \cite[Chapter 6]{axiomatic}.

In the usual stable homotopy category of spectra, though, the homotopy
groups of the sphere form a non-noetherian ring, so Neeman's work is
not as relevant.  So part of the motivation for this paper is that, in
order to gain insight into spectra, one should study the derived
categories of non-noetherian rings from the stable homotopy theoretic
point of view.  Furthermore, since the prime ideal spectrum of the
stable homotopy groups of spheres is pretty small -- when working
$p$-locally, there are only two prime ideals in $\pi_{*}(S^{0})$ --
one should study non-noetherian rings with few prime ideals; hence the
ring $\Alg$ under consideration here.  One might hope that results
about $\derived{\Alg}$ might provide insight into part of the
stable homotopy category of spectra, or at least into part of the
Bousfield lattice for that category.

The paper \cite{hovey-palmieri;bousfield} asks a number of questions
about the Bousfield lattice for the category of spectra, but at least
some of those questions make sense in any stable homotopy category.
So one goal of this paper was to try to answer those questions.  This
has been somewhat successful; in this paper we settle the ``retract
conjecture'' \cite[Conjecture 3.12]{hovey-palmieri;bousfield} -- we
use Theorem~\ref{thm-tensor-height} to show in
Corollary~\ref{cor-retract} that this fails in $\derived{\Alg}$ -- and
the ``Dichotomy conjecture'' \cite[Conjecture
7.5]{hovey-palmieri;bousfield} -- we show in Section~\ref{sec-minimal}
that this holds in $\derived{\Alg}$.  Many of the other conjectures
deal with specific spectra for which there are no good analogues in
$\derived{\Alg}$, and so those are not relevant.  Of the relevant
ones, the main outstanding one is \cite[Conjecture
9.1]{hovey-palmieri;bousfield}: is every localizing subcategory of
$\derived{\Alg}$ a Bousfield class?  (Localizing subcategories are
defined in Definition~\ref{defn-thick}.)  Questions involving
localizing subcategories seem to be hard; for example, it is not even
known whether there is a set of localizing subcategories.  It would be
nice to have more information about this conjecture in the category
$\derived{\Alg}$, or in derived categories in general.  We also
mention two other questions about the derived category of a
commutative ring $R$: Question~\ref{question-bousfield} below asks
whether any object in $\derived{R}$ is Bousfield equivalent to an
$R$-module.  Question~\ref{question-set} asks whether the Bousfield
lattice $\B (R)$ for $\derived{R}$ always forms a set.

\subsection{Organization}

The structure of the paper is as follows: in
Section~\ref{sec-notation}, we set up some basic notation for use
throughout the rest of the paper; in particular, for any subset $S$ of
$\N = \{1,2,3,\dotsc\}$, we examine the $\Alg$-module
\[
\M{S} := k[x_{i} : i \in S] / (x_{i}^{n_{i}})
\]
and its vector space dual $\I{S}$.  A homotopy theorist might view
$\M{S}$ as a sort of Moore spectrum, at least if $n_{i}=2$ for each
$i$, while $\I{S}$ is its Brown-Comenetz dual.  In
Section~\ref{sec-shapiro}, we set up some homological algebra, and we
use it to prove that $\M{T} \Smash \I{U}=0$ if and only if $T\mycomp
\cap U$ is infinite.  In Section~\ref{sec-bousfield}, we use this
result to find the relationships among the Bousfield classes for these
objects; for example, we show that if $S \subseteq T$ then
$\bousfield{\M{S}} \leq \bousfield{\M{T}}$ and $\bousfield{\I{T}} \leq
\bousfield{\I{S}}$.  We use these kinds of computations to prove
Theorem~\ref{thm-uncountable}, and we also show that $\I{S} \Smash
\I{S}=0$ whenever $S$ is infinite.

In Section~\ref{sec-more-bousfield}, we examine other objects
constructed from $\M{S}$ and $\I{T}$, and discuss the sublattice
$\latt$ discussed above.  We prove Theorem~\ref{thm-tensor-height} in
Section~\ref{sec-tensor-height}.  We explore the minimality of
$\bousfield{\I{N}}$, and some related results, in
Section~\ref{sec-minimal}.  Finally, in Section~\ref{sec-nilp} we
prove a nilpotence theorem, a thick subcategory theorem, and a
periodicity theorem.

We point out that Neeman \cite{neeman;oddball} has considered a
similar situation: he studied the Bousfield classes in the derived
category of the ring
\[
k[x_{2}, x_{3}, x_{4}, \dots] / (x_{2}^{2}, x_{3}^{3}, x_{4}^{4}, \dots),
\]
where $k$ is a field.  He showed that there were at least
$2^{2^{\aleph_{0}}}$ Bousfield classes for this ring as well.  His
methods were a bit different, and his goal was to demonstrate the size
of the Bousfield lattice, not to investigate its finer properties, as
we are trying to do here.

Acknowledgments: Mark Hovey, James Zhang, and Paul Smith have provided
valuable input into various aspects of this research.

\section{Notation and basic algebra}\label{sec-notation}

All rings in this paper are graded, as are all modules over them.  The
\emph{suspension} or \emph{shift} functor $\Sigma$ on graded modules
is defined as follows: given a graded module $M=\bigoplus_{m} M_{m}$,
the module $\Sigma M$ has $m$th graded piece $M_{m-1}$.  Iterating
this, $\Sigma^{j} M$ has $m$th graded piece $M_{m-j}$ for any integer
$j$.

All chain complexes in this paper are bigraded, with the first grading
the homological one, and the second grading induced by that on the
modules.  The differential on such a complex will always have degree
$(1,0)$.  Given integers $i$ and $j$, the $(i,j)$-\emph{suspension}
functor $\Sigma^{i,j}$ is defined as follows: given a graded ring $R$
and a chain complex $X=(\dots \rightarrow X^{n} \rightarrow X^{n+1}
\rightarrow \dots)$ of graded $R$-modules, $\Sigma^{i,j}X$ is the
chain complex with $n$th term $\Sigma^{j} X^{n-i}$.  (Note that
$\Sigma$ with a single grading means the shift functor on modules,
while with a double grading it is the shift functor on complexes.)

Fix a countable field $k$ and a sequence of integers $n_{i} \geq 2$
for $i \geq 1$.  As in \eqref{eqn-trunc}, we define the $k$-algebra
$\Alg$ by
\[
\Alg = k[x_{1}, x_{2}, x_{3}, \dotsc] / (x_{i}^{n_{i}} \ \text{for all} \ i)
\]
with $\deg x_{i}=2^{i}$.  The unbounded derived category
$\derived{\Alg}$ is a stable homotopy category, and more precisely a
monogenic stable homotopy category, and since $k$ is countable, a
Brown category; see \cite[Definitions 1.1.4 and 4.1.4]{axiomatic} for
the definitions and \cite[Section 9.3]{axiomatic} for the proof that
$\derived{\Alg}$ has these properties.  From the stable homotopy
theoretic point of view, the derived tensor product is the \emph{smash
product} in $\derived{\Alg}$, and the rank one free module $\Alg$ is
the \emph{sphere object}.

\begin{notation}\label{notation-MIT}
Let $\N$ be the set of positive integers and let $S$ be a subset of
$\N$.  Let $\alg{S}$ be the subalgebra of $\Alg$ generated by
$x_{i}$ with $i \in S$; we also write $S\mycomp$ for the complement of
$S$ in $\N$.
\end{notation}

Then for any $S \subseteq \N$ there is an algebra isomorphism $\Alg
\cong \alg{S} \otimes \alg{S\mycomp}$,\footnote{Any unadorned
tensor product in this paper is over the base field $k$.}
which leads to two algebra homomorphisms: the inclusion map
$\alg{S} \rightarrow \Alg$ and the projection map $\Alg
\rightarrow \alg{S}$.  These, respectively, lead to two useful
functors: the \emph{restriction} functor $\restrict{-}{S}:
\Alg\text{-}\Mod \rightarrow \alg{S}\text{-}\Mod$, and the
\emph{inclusion} functor $\incl{S}{}: \alg{S}\text{-}\Mod
\rightarrow \Alg\text{-}\Mod$.  More generally, if $S \subseteq T
\subseteq \N$, we have functors
\begin{gather*}
\restrict{-}{S}^{T}: \alg{T}\text{-}\Mod \rightarrow \alg{S}\text{-}\Mod,\\
\incl{S}{T}: \alg{S}\text{-}\Mod \rightarrow \alg{T}\text{-}\Mod.
\end{gather*}
These also induce functors on the corresponding derived categories; we
write those as $\restrict{-}{S}^{T}$ and $\incl{S}{T}$ as well.  For
any object $X$ in $\derived{\alg{S}}$, we have
$\restrict{(\incl{S}{T}X)}{S}^{T} \cong X$.

\begin{notation}
Fix $S \subseteq T \subseteq \N$.
\begin{enumerate}
\item We make $\M{S}$ into a $\alg{T}$-module by applying the
inclusion functor $\incl{S}{T}$ to a rank one free $\alg{S}$-module;
then we have isomorphisms of $\alg{T}$-modules
\[
\M{S} \cong \alg{T} /(x_{i} : i \not \in S) \cong \alg{T}
\otimes_{\alg{T-S}} k.
\]
Note that this makes $\M{S}$ into a ring object in
$\derived{\alg{T}}$.
\item As a $\alg{T}$-module, let $\I{S} = \Hom_{k}^{*}(\M{S},k)$
be the graded $k$-dual of $\M{S}$, concentrated in non-positive
degrees.
\end{enumerate}
\end{notation}

Note that if $S$ is finite, then $\I{S}$ is isomorphic to $\M{S}$, up
to suspension.

\section{Shapiro's lemma and Brown-Comenetz duality}\label{sec-shapiro}

In this section, we state Shapiro's lemma for $\derived{\Alg}$, we
define Brown-Comenetz duality in this setting, and we combine these
two concepts to compute some tensor products.
Lemma~\ref{lemma-smash}, which says that $\M{T} \Smash \I{U}=0$
if and only if $T\mycomp \cap U$ is infinite, is used throughout
the paper.

We start with some basic homological algebra.

\begin{lemma}[Shapiro's lemma]
\label{lemma-shapiro}
Fix a subset $S$ of $\N$.
\begin{enumerate}
\item \textup{[Homology]}
For any object $X$ in $\derived{\Alg}$, there is an isomorphism 
\[
X \Smash \M{S\mycomp} \cong \incl{S}{}\left(\restrict{X}{S}
\ltensor{\alg{S}} k\right).
\]
\item \textup{[Cohomology]} For any object $X$ in $\derived{\Alg}$,
there is an isomorphism 
\[
\Hom_{\derived{\Alg}}(\M{S\mycomp}, X)^{*} \cong
\Hom_{\derived{\alg{S}}}^{*} (k,\restrict{X}{S}).
\]
\end{enumerate}
\end{lemma}

\begin{proof}
Both parts are standard.  For part (a), for example, we have
\[
X \ltensor{\Alg} \M{S\mycomp} \cong
X \ltensor{\Alg} (\restrict{\Alg}{S} \underset{\alg{S}}
{\otimes} k) \cong X \ltensor{\Alg} (\restrict{\Alg}{S}
\ltensor{\alg{S}} k) \cong \incl{S}{}\left(\restrict{X}{S}
\ltensor{\alg{S}} k\right).
\]
Part (b) is similar.
\end{proof}

Here is one Ext calculation for later use.

\begin{lemma}\label{lemma-ext}
Fix $T \subseteq S \subseteq \N$.  If $T$ is infinite, then for any
set $J$,
\[
\Ext_{\alg{S}}^{q}(k,\bigoplus_{J} \M{T})=0 \quad \text{for all}
\quad  q \geq 0.
\]
\end{lemma}

\begin{proof}
We compute the Ext groups in question with the spectral sequence
associated to the central algebra extension
\[
k \rightarrow \alg{T} \rightarrow \alg{S} \rightarrow \alg{T-S}
\rightarrow k.
\]
For any $\alg{S}$-module $N$, there is a strongly convergent
spectral sequence \cite[Theorem XVI.6.1]{cartan-eilenberg}
\[
\Ext_{\alg{S-T}}^{p} (k, \Ext_{\alg{T}}^{q}(k,N)) \Rightarrow
\Ext_{\Alg}^{p+q} (k,N).
\]
We apply this with $N=\bigoplus \M{T}$.  We claim that the groups 
\begin{equation}\label{eqn-ext}
\Ext_{\alg{T}}^{q}(k,\bigoplus_{J} \M{T})
\end{equation}
are zero for all $q$.

Since $T$ is infinite, then $\alg{T}$ is what Margolis calls a
$P$-algebra \cite[Chapter 13]{margolis}.  Therefore by \cite[Theorem
13.12]{margolis}, any free bounded below $\alg{T}$-module is also
injective.  In particular, $\bigoplus \M{T}$ is injective, so the
above Ext group \eqref{eqn-ext} is zero when $q>0$.

Also because $T$ is infinite, there are no $\alg{T}$-module maps
$k \rightarrow \bigoplus \M{T}$, and we see that the above Ext group
\eqref{eqn-ext} is zero when $q=0$.
\end{proof}

Now we move on to Brown-Comenetz duality.  See \cite{brown-comenetz}
and \cite{ravenel;localization} for some information about this
construction in stable homotopy.  The analogue here is vector space
duality: for any object $Y$ in $\derived{\Alg}$, we define its
\emph{Brown-Comenetz dual} to be the object $\bcdual{Y} = \RHom_{k}
(Y,k)$ in $\derived{\Alg}$.  For example, the
Brown-Comenetz dual of $\M{S}$ is $\bcdual{\M{S}}=\I{S}$.

\begin{lemma}\label{lemma-bcdual}
For any object $X$ in $\derived{\Alg}$, we have 
\[
\Hom_{k}^{*} (H^{*}(X \Smash Y), k) \cong
\Hom_{\derived{\Alg}}^{*}(X, \bcdual{Y}).
\]
Thus if $Y$ is locally finite and 
$\RHom_{\Alg}^{*}(X,Y) = 0$, then $X \Smash
\bcdual{Y} = 0$.
\end{lemma}

\begin{proof}
The first statement follows from tensor-hom adjointness: in the
category of $\Alg$-modules, we have
\[
\Hom_{\Alg} (X, \bcdual{Y}) = \Hom_{\Alg} (X, \Hom_{k}(Y,k))
\cong \Hom_{k} (X \tensor{\Alg} Y, k).
\]
This induces the above isomorphism in the derived category.  

The second statement now follows: since $Y$ is locally finite, it is
self-double-dual: $I(\bcdual{Y}) \cong Y$.  First, 
$\RHom_{\derived{\Alg}}^{*}(X,Y) = 0$ if and only if
$\Hom_{\derived{\Alg}}^{*}(X,I(\bcdual{Y}))=0$, and then by the first 
part, we have
\[
\Hom_{k} (H^{*}(X \Smash \bcdual{Y}), k) = 0.
\]
This implies that $H^{*}(X \Smash \bcdual{Y}) = 0$, which means that
$X \Smash \bcdual{Y}=0$.
\end{proof}

\begin{corollary}\label{cor-bcdual}
In particular, $\RHom_{\Alg}^{*}(X,\M{S}) = 0$ if and
only if $X \Smash \I{S} = 0$.
\end{corollary}

Now we start computing some tensor products.

\begin{lemma}\label{lemma-smash}
Let $T$ and $U$ be subsets of $\N$.  Then $\M{T} \Smash \I{U} = 0$ if
and only if $T\mycomp \cap U$ is infinite.
\end{lemma}

\begin{proof}
By Corollary~\ref{cor-bcdual}, we have 
\[
\M{T} \Smash \I{U}=0 \quad \Longleftrightarrow \quad 
\Hom_{\derived{\Alg}}^{*}(\M{T},\M{U})=0.
\]
According to Shapiro's lemma \ref{lemma-shapiro}, 
\[
\Hom_{\derived{\Alg}}^{*}(\M{T},\M{U}) 
\cong \Hom_{\derived{\alg{T\mycomp}}}^{*} (k,
\restrict{\M{U}}{T\mycomp}),
\]
and as a $\alg{T\mycomp}$-module, $\M{U}$ is isomorphic to 
\[
\M{U \cap T\mycomp} \otimes \M{U \cap T} \cong \bigoplus \M{U \cap
T\mycomp},
\]
where the direct sum is indexed by a basis for $\M{U \cap T}$, since
$\M{U \cap T}$ is trivial as a $\alg{T\mycomp}$-module.  So we
want to understand whether 
\[
\Hom_{\derived{\alg{T\mycomp}}}^{*} (k, \bigoplus \M{U \cap T})
\]
is zero.  If $U \cap T$ is infinite, then Lemma~\ref{lemma-ext} shows
that this group is zero.  If $U \cap T$ is finite, then $\M{U \cap T}$
is a finite-dimensional vector space, and there are
$\alg{T\mycomp}$-module maps $k \rightarrow \M{U \cap T}$ and thus
nonzero maps $k \rightarrow \bigoplus \M{U \cap T}$; thus this group is
nonzero.
\end{proof}

\section{Beginning Bousfield class computations}\label{sec-bousfield}

Now we use Lemma~\ref{lemma-smash} to get information about the
Bousfield class of $\M{S}$ for subsets $S \subseteq \N$.

We start by constructing some useful exact triangles.  Recall that
$n_{i}$ is the nilpotence height of the generator $x_{i} \in \Alg$.

\begin{lemma}\label{lemma-exact}
Fix subsets $S \subsetneq T \subseteq \N$ and an integer $i \in T-S$.
There are objects $M_{j}$, $1 \leq j \leq n_{i}$, with $M_{1} = \M{S}$
and $M_{n_{i}} \cong \M{S \cup \{i\}}$, and $n_{i}-1$
exact triangles in $\derived{\alg{T}}$
\[
\Sigma^{1,0} \M{S} \rightarrow \Sigma^{0,\deg
x_{i}} M_{j-1} \rightarrow M_{j}, \quad 2 \leq j \leq n_{i}.
\]
\end{lemma}

\begin{proof}
There are exact sequences of $\alg{i}$-modules
\begin{gather*}
0 \rightarrow \Sigma^{\deg x_{i}} k \rightarrow k[x_{i}]/(x_{i}^{2})
\rightarrow k \rightarrow 0, \\
0 \rightarrow \Sigma^{\deg x_{i}} k[x_{i}]/(x_{i}^{2}) \rightarrow
k[x_{i}]/(x_{i}^{3}) \rightarrow k \rightarrow 0, \\
\vdots \\ 
0 \rightarrow \Sigma^{\deg x_{i}} k[x_{i}]/(x_{i}^{j-1}) \rightarrow
k[x_{i}]/(x_{i}^{j}) \rightarrow k \rightarrow 0, \\
\vdots \\
0 \rightarrow \Sigma^{\deg x_{i}} k[x_{i}]/(x_{i}^{n_{i}-1})
\rightarrow \alg{i} \rightarrow k \rightarrow 0.
\end{gather*}
So define $M_{j}$ by 
\[
M_{j} = \alg{S} \ltensor{\alg{i}} k[x_{i}]/(x_{i}^{j}).
\]
Then applying $\alg{S \cup \{i\}} \otimes_{\alg{i}} -$ to these exact
sequences gives short exact sequences of $\alg{T}$-modules, and hence
the desired exact triangles.
\end{proof}

\begin{definition}\label{defn-thick}
\begin{enumerate}
\item 
Given subsets $S$ and $T$ of $\N$, we say that $T$ is \emph{cofinite}
in $S$ if $T \subseteq S$ and the complement of $T$ in $S$ is finite.
We say that subsets $S$ and $T$ of $\N$ are \emph{commensurable} if $S
\cap T$ is cofinite in both $S$ and $T$.  We write $S \comm T$ if $S$
and $T$ are commensurable.  We write $S \lecomm T$ if $S$ is
commensurable with a subset of $T$.
\item 
A full subcategory of $\derived{\Alg}$ is a \emph{localizing
subcategory} if it is triangulated and closed under arbitrary direct
sums (and hence arbitrary colimits).  For any object $Y$ of
$\derived{\Alg}$, the localizing subcategory \emph{generated by
$Y$} is the intersection of all of the localizing subcategories
containing $Y$.  Given two objects $X$ and $Y$ of $\derived{\Alg}$,
we say that $X$ \emph{can be built from} $Y$ if $X$ is in the
localizing subcategory generated by $Y$.
\item 
A subcategory $\cat{D}$ of $\derived{\Alg}$ is called \emph{thick}
if it is triangulated and ``closed under retracts'': if $X \oplus Y$
is in $\cat{D}$, then so are $X$ and $Y$.  For any object $Y$ of
$\derived{\Alg}$, we write $\thick (Y)$ for the thick subcategory
generated by $Y$ -- that is, the intersection of all of the thick
subcategories containing $Y$.
\end{enumerate}
\end{definition}

Note that any localizing subcategory is thick by the Eilenberg swindle
\cite[Lemma 1.4.9]{axiomatic}, but thick subcategories need not be
localizing.  Also, it is well known that in the derived category of
any commutative ring $R$, every object can be built from $R$; see
\cite[Theorem 9.3.1]{axiomatic}, for example.

We recall from the introduction the definition of the \emph{Bousfield
class} $\bousfield{X}$ of an object $X$:
\[
\bousfield{X} = \{W \in \ob \derived{R} \suchthat X \ltensor{R} W = 0 \}.
\]
These are partially ordered by reverse inclusion.

\begin{lemma}\label{lemma-build-M}
Fix subsets $S$ and $T$ of $\N$. 
\begin{enumerate}
\item $\M{S}$ is in the thick subcategory generated by $\M{T}$ if and
only if $T$ is cofinite in $S$.
\item The following are equivalent:
\begin{enumerate}
\item $S \lecomm T$.
\item $\M{S}$ can be built from $\M{T}$.
\item $\bousfield{\M{S}} \leq \bousfield{\M{T}}$.
\end{enumerate}
\end{enumerate}
\end{lemma}

\begin{proof}
(a) First assume that $T$ is cofinite in $S$; we want to show that
$\M{S}$ is in $\thick (\M{T})$.  It is enough to assume that $S = T
\cup \{i \}$, in which case the exact triangles in
Lemma~\ref{lemma-exact} do the job.

Now assume that $T$ is not cofinite in $S$.  There are two cases to
consider: either $T$ is not a subset of $S$, or $T \subseteq S$ and
$S-T$ is infinite.  Since
\[
\Hom_{\derived{\alg{T}}}^{*}(\M{T},k) \cong
\Ext_{\alg{T}}^{*}(\M{T},k),
\]
and this is isomorphic to $k$ concentrated in degree 0, then
$\Hom_{\derived{\alg{T}}}^{*}(\restrict{X}{T}, k)$ is
finite-dimensional for any object $X$ in $\thick (\M{T})$.  When
viewed as a $\alg{T}$-module, $\M{S}$ splits as a direct sum:
\[
\restrict{\M{S}}{T} \cong \restrict{\M{S-T}}{T} \otimes \restrict{\M{T
\cap S}}{T} \cong \bigoplus \restrict{\M{T \cap S}}{T},
\]
where the direct sum is indexed by a basis for $\M{S-T}$, since
$\M{S-T}$ is trivial as a $\alg{T}$-module.  If $S-T$ is infinite,
then so is this direct sum, and hence so is
$\Ext_{\alg{T}}^{0}(\restrict{\M{S}}{T},k)$.  Therefore $\M{S}$ cannot
be in $\thick (\M{T})$.  Otherwise, Shapiro's
lemma~\ref{lemma-shapiro} says that
\[
\Ext_{\alg{T}}^{*}(\M{T\cap S},k) \cong \Ext_{\alg{T-T\cap 
S}}^{*}(k,k),
\]
and this is infinite-dimensional as long as $T \neq T \cap S$, and so,
again, $\M{S}$ cannot be in $\thick (\M{T})$.

(b) To show that $\M{S}$ can be built from $\M{T}$ if $S \lecomm T$,
we need to prove that $\M{S}$ can be built from $\M{T}$ if either $S =
T \cup \{i\}$ or if $S \subseteq T$.  If $S=T \cup \{i\}$, then part
(a) shows us that in this case $\M{S}$ is in $\thick (\M{T})$ and
hence can be built from $\M{T}$.

Now we assume that $S \subseteq T$.  First we work in the category
$\derived{\alg{T}}$.  Every object in $\derived{\alg{T}}$ may
be built from $\M{T}$, and in particular $\restrict{\M{S}}{T}$ may be
built from $\M{T}$.  Now apply the inclusion functor $\incl{T}{}:
\derived{\alg{T}} \rightarrow \derived{\Alg}$: this functor is
the identity on objects, so it is exact and commutes with direct
limits; thus in $\derived{\Alg}$, $\M{S}$ may be built from
$\M{T}$.  This finishes the proof that (1) implies (2).

(One can make this sort of argument more explicit: if $P_{*}$ is a projective
resolution of $\M{S}$ as a $\alg{T}$-module, then $\M{S}$ is the
colimit of the truncations
\[
0 \rightarrow P_{n} \rightarrow \dotsb \rightarrow P_{0} \rightarrow 0
\]
of $P_{*}$.  One can then show by a simple induction argument that
each of these may be built from $\alg{T}$.)

Since the derived tensor product is exact and commutes with direct
limits, if $X$ can be built from $Y$, then $\bousfield{X} \leq
\bousfield{Y}$.  Thus (2) implies (3).

Now we prove that (3) implies (1).  Assume that $S$ is not
commensurable with any subset of $T$; we will show that
$\bousfield{\M{S}}$ is not less than or equal to $\bousfield{\M{T}}$.
If $S \not \lecomm T$, then in particular $S$ is not commensurable
with $S\cap T$, so $S \cap T$ is not cofinite in $S$.  If we let
$U=S-(S\cap T)$, then we see that $U$ is an infinite subset of
$T\mycomp$ which does not intersect $S\mycomp$.  So by
Lemma~\ref{lemma-smash}, we have
\[
\I{U} \Smash \M{T} = 0 \quad \text{but} \quad \I{U} \Smash \M{S} \neq
0.
\]
Thus $\bousfield{\M{T}}$ is not greater than or equal to
$\bousfield{\M{S}}$, as desired.
\end{proof}

\begin{corollary}\label{cor-build-k}
$k=\M{\emptyset}$ may be built from $\M{S}$ for any $S \subseteq \N$.
\end{corollary}

\begin{corollary}\label{cor-uncountable}
For subsets $S$ and $T$ of $\N$, $\bousfield{\M{S}} =
\bousfield{\M{T}}$ if and only if $S \comm T$.  Thus there are
uncountably many Bousfield classes in $\derived{\Alg}$.  Indeed,
there is an uncountable antichain in the Bousfield lattice.
\end{corollary}

Recall that an \emph{antichain} in a partially ordered set is a
subset any two of whose elements are not comparable.

\begin{proof}
For each prime number $p$, let $P_{p}=\{p^{k} \suchthat k \in N \}$.
We claim that there is an uncountable set $\{S_{\alpha}\}$ of
subsets of $\N$ such that
\begin{enumerate}
\item [(i)] the sets $S_{\alpha} \cap P_{p}$ and $S_{\alpha}\mycomp
\cap P_{p}$ are infinite for all $\alpha$, $p$, and
\item [(ii)] the commensurability classes of $S_{\alpha}$ and
$S_{\beta}$ are not comparable, for all $\alpha$, $\beta$.
\end{enumerate}
First, there are infinitely many such subsets; for instance, for each
prime $q$, define $S_{q}$ to be $\bigcup_{p} \{p^{kq} \suchthat k \in
\N \}$.

Now suppose that $S_{2}$, $S_{3}$, $S_{5}$, \dots are any subsets of
$N$, indexed by the prime numbers, satisfying (i) and (ii) above.  We
construct a set $T$, satisfying (i), whose commensurability class is
not comparable to that of any $S_{p}$.  Let
\[
T = \bigcup_{p} \left( P_{p} - (S_{p} \cap P_{p}) \right).
\]
Then for each $p$, the sets $T \cap P_{p}$ and $S_{p} \cap P_{p}$ form
a partition of $P_{p}$ into two infinite sets, and hence $T$ is as
advertised.  Therefore the collection of subsets of $\N$ satisfying
(i) and (ii) is not countable.

Finally, given an uncountable set $\{S_{\alpha}\}$ of subsets of $\N$
satisfying (i) and (ii), the set of Bousfield classes
$\{\bousfield{\M{S_{\alpha}}}\}$ is uncountable, and no two elements
of it are comparable.
\end{proof}

\begin{remark}\label{remark-uncountable}
The uncountable antichain in the corollary consists of Bousfield
classes of $\alg{S_{\alpha}}$ for various subsets $S_{\alpha}$ of
$\N$.  A simple modification in the proof yields this result: for any
subsets $T$, $U$ of $\N$ with $T \lecomm U$ and $T \not \comm U$,
there is an uncountable antichain in the interval between
$\bousfield{\alg{T}}$ and $\bousfield{\alg{U}}$: one may choose each
$S_{\alpha}$ so that $T \lecomm S_{\alpha} \lecomm U$.
\end{remark}

\begin{theorem}\label{thm-big}
The lattice $\B$ has cardinality at least $2^{2^{\aleph_{0}}}$.
\end{theorem}

\begin{proof}
By Corollary~\ref{cor-uncountable}, $\B$ contains an uncountable
antichain; indeed, the elements of the antichain may be chosen to be
$\M{S_{\alpha}}$, where $\{S_{\alpha}\}_{\alpha \in J}$ is an
uncountable set of subsets of $\N$, no two of which have comparable
commensurability classes.  Here, $J$ is some indexing set with
cardinality at least $2^{\aleph_{0}}$.  For any subset $I$ of $J$, let
\[
X_{I} = \bigoplus_{\alpha \in I} \M{S_{\alpha}}.
\]
We claim that the objects $X_{I}$ have distinct Bousfield classes: if
$I \neq I'$, then $\bousfield{X_{I}} \neq \bousfield{X_{I'}}$.  If $I
\neq I'$, then either $I - I'$ or $I' - I$ is nonempty; without loss
of generality, suppose that $\alpha \in I - I'$.  Then the
commensurability class of $S_{\alpha}$ is not comparable with
$S_{\beta}$ for any $\beta \in I'$; therefore $\I{S_{\alpha}} \Smash
\M{S_{\beta}}=0$ for every $\beta \in I'$, so $\I{S_{\alpha}} \Smash
X_{I'} = 0$.  On the other hand, $\I{S_{\alpha}} \Smash \M{S_{\alpha}}
\neq 0$, so $\I{S_{\alpha}} \Smash X_{I} \neq 0$.  (Indeed, one can
recover the set $I$ from $\bousfield{X_{I}}$: $I$ will consist of
those indices $\beta$ for which $X_{I} \Smash \I{S_{\beta}} \neq 0$.)
This completes the proof.
\end{proof}

Now we examine the dual picture and determine when $\I{S}$ may be
built from $\I{T}$.

\begin{lemma}\label{lemma-build-I}
Fix subsets $S$ and $T$ of $\N$.  
\begin{enumerate}
\item 
$\I{S}$ is in $\thick (\I{T})$ if and only if $T$ is cofinite in $S$.
\item The following are equivalent:
\begin{enumerate}
\item $T \lecomm S$.
\item $\I{S}$ can be built from $\I{T}$.
\item $\bousfield{\I{S}} \leq \bousfield{\I{T}}$.
\end{enumerate}
\end{enumerate}
\end{lemma}

\begin{proof}
(a) If $S = T \cup \{i\}$, then we dualize the argument in
Lemma~\ref{lemma-build-M}(a): the Brown-Comenetz duals of the exact
triangles in Lemma~\ref{lemma-exact} show that $\I{S}$ is in $\thick
(\I{T})$.

Now assume that $T$ is not cofinite in $S$.  As in
Lemma~\ref{lemma-build-M}(a), we may assume that either $T$ is not a
subset of $S$ or $T \subseteq S$ with $S-T$ infinite.  Since
$\Hom_{\derived{\alg{T}}}^{*}(k,\I{T})$ is a copy of $k$ concentrated
in dimension $0$ (by Lemma~\ref{lemma-bcdual}), then
$\Hom_{\derived{\alg{T}}}^{*}(k,\restrict{X}{T})$ is
finite-dimensional for all $X$ in $\thick (\I{T})$.  When viewed as a
$\alg{T}$-module, $\I{S}$ splits as a direct sum:
\[
\restrict{\I{S}}{T} \cong \restrict{\I{S-T}}{T} \otimes \restrict{\I{T
\cap S}}{T} \cong \bigoplus \restrict{\I{T \cap S}}{T},
\]
as in the proof of Lemma~\ref{lemma-build-M}.  If $S-T$ is
infinite-dimensional, then the direct sum is infinite, so
$\Hom_{\derived{\alg{T}}}^{0}(k, \restrict{\I{S}}{T})$ is
infinite-dimensional.  Otherwise, by Lemmas~\ref{lemma-shapiro} and
\ref{lemma-bcdual}, we have
\begin{align*}
\Hom_{\derived{\alg{T}}}^{*} (k, \I{T \cap S})
& \cong \Hom_{k}^{*} (H^{*}(k \ltensor{\alg{T}} \M{T \cap S}), k) \\
& \cong \Hom_{k}^{*} (H^{*}(k \ltensor{\alg{U}} k), k),
\end{align*}
where $U=T-T \cap S$.
The cohomology of $(k \ltensor{\alg{U}} k)$ is
$\Tor^{\alg{U}}_{*}(k,k)$, which is infinite-dimensional
as long as $T \neq T \cap S$.

(b) 
As in the proof of Lemma~\ref{lemma-build-M}(b), we show that if either
$T = S \cup \{i\}$ or if $T \subseteq S$, then $\I{S}$ can be built
from $\I{T}$.  First assume that $T=S \cup \{i\}$.  Then $\I{T} \cong
\I{S} \otimes \I{i}$, and so up to suspension, $\I{T} \cong \I{S}
\otimes \M{x_{i}}$.  By Lemma~\ref{lemma-build-M}, $k$ may be built
from $\M{x_{i}}$, and so tensoring with $\I{S}$, we find that $\I{S}$
may be built from $\I{T}$.

Now assume that $T \subseteq S$.  In the special case when $T =
\emptyset$, we can use the Postnikov tower for $\I{S}$ to prove this
-- see Lemma~\ref{lemma-build-I-from-k} below.  For general $T
\subseteq S$, we first note that in the category
$\derived{\alg{S-T}}$, $\I{S-T}$ may be built from $k$, by the
previous sentence.  Now tensor with $\I{T}$: as $\alg{S} = \alg{S-T}
\otimes \alg{T}$-modules, $\I{S-T} \otimes \I{T}$ may be built from $k
\otimes \I{T}$.  Finally, we apply the inclusion functor from
$\alg{S}$-modules to $\Alg$-modules.  This proves that (1) implies
(2).

As explained in the proof of Lemma~\ref{lemma-build-M}, (2) always
implies (3).

Finally, assume that $T$ is not commensurable with any subset of $S$.
This means that $U:=S \cap T$ is not cofinite in $T$.  Thus $U$ is an
infinite set with $U \cap S$ empty and $U \cap T$ infinite.  Therefore
$\I{S} \Smash \M{U\mycomp} \neq 0$ while $\I{T} \Smash \M{U\mycomp} =
0$.  Thus (3) implies (1).
\end{proof}

We used the following special case of the preceding lemma in its
proof.

\begin{lemma}\label{lemma-build-I-from-k}
If $S$ is any subset of $\N$, $\I{S}$ can be built from
$k=\I{\emptyset}$.
\end{lemma}

\begin{proof}
The stable homotopy theoretic version of the proof is this: use the
Postnikov tower for $\I{S}$.  The ``homotopy groups'' (= homology
groups in $\derived{\Alg}$) for $\I{S}$ are zero except in dimension
zero, and $\pi_{0}(\I{S}) = H^{0}(\I{S})$ is the graded vector space dual of $\alg{S}$.
When $S$ is finite, this is finite, and so the object may be built
from $k$.  When $S$ is infinite, this is infinite but is bounded
above, so the Postnikov tower describes how to write it as a colimit
of objects built from $k$.

A more explicit, homological algebra proof is as follows: first, write
$S=\{x_{i_{1}}, x_{i_{2}}, \dots \}$.  Assume that $n_{i_{k}}>2$ for
each $k$; the case when $n_{i_{k}}=2$ is similar but easier.  Whenever
$T$ is finite, $\I{T}$ is a shifted copy of $\alg{T}$; more precisely,
$\I{i} \cong \Sigma^{0,-(n_{i}-1)\deg x_{i}}\M{i}$, and when $S$ and
$T$ disjoint, $\I{S \cup T} \cong \I{S} \otimes \I{T}$, so when $T$ is
finite, we have
\[
\I{T} \cong \Sigma^{0, -\sum_{j \in T} (n_{j}-1) \deg x_{j}} \M{T}.
\]
Note also that by Lemma~\ref{lemma-build-M}, $\M{i}$ is in the thick
subcategory generated by $k$, and so may be built from $k$; therefore
the same is true of $\I{i}$.

Now for any $i$, the exact triangles in Lemma~\ref{lemma-exact} give
maps
\[
\begin{split}
k \rightarrow \Sigma^{0,-\deg x_{i}} M_{2} 
\rightarrow \Sigma^{0,-2\deg x_{i}} M_{3}
\rightarrow \dots \\
\rightarrow
\Sigma^{0,-(n_{i}-2) \deg x_{i}} M_{n_{i}-1}
\rightarrow \Sigma^{0,-(n_{i}-1)\deg x_{i}}
\M{i},
\end{split}
\]
which we compose to get $k \xrightarrow{f_{i}} \I{i}$.
Given two integers $i$ and $j$ in $\N$, we can form
\[
k \xrightarrow{f_{i}} \I{i} \xrightarrow{1 \otimes f_{j}} \I{i}
\otimes \I{j} \cong \I{i,j}.
\]
In this manner, we get a composite
\[
k \rightarrow \I{i_{1}} \rightarrow \I{i_{1}, i_{2}} \rightarrow \dots
\]
which displays $\I{S}$ as a colimit of objects built from $k$.
\end{proof}

We can then dualize the arguments in Corollary~\ref{cor-uncountable}
and Remark~\ref{remark-uncountable} to get the following.

\begin{corollary}\label{cor-I-uncountable}
For subsets $S$ and $T$ of $\N$, $\bousfield{\I{S}} =
\bousfield{\I{T}}$ if and only if $S \comm T$.  Thus for subsets $S$
and $T$ of $\N$ with $T \lecomm S$ and $T \not \comm S$, there is an
uncountable antichain in the interval between $\bousfield{\I{S}}$ and
$\bousfield{\I{T}}$.
\end{corollary}

As another application, we have another tensor product computation.

\begin{proposition}\label{prop-I-smash-I}
Fix subsets $S$ and $T$ of $\N$.  If $S$ is infinite, then $\I{S}
\Smash \I{T} = 0$.
\end{proposition}

For example, if $S$ is infinite, then $\I{S} \Smash \I{S} = 0$ and
$\I{S} \Smash k=0$.

\begin{proof}
For fixed $S$, the collection of objects $X$ so that $\I{S} \Smash
X=0$ forms a localizing subcategory.  By Lemma~\ref{lemma-smash},
since $S$ is infinite, this subcategory contains $\M{S\mycomp}$, and
therefore it contains $k$ and $\I{T}$, by Corollary~\ref{cor-build-k}
and Lemma~\ref{lemma-build-I-from-k}.  Thus $\I{S} \Smash \I{T} = 0$.
\end{proof}

If $S$ and $T$ are both finite, then
$\bousfield{\I{S}}=\bousfield{\I{T}}=\bousfield{k}$, and so $\I{S}
\Smash \I{T} \neq 0$.  Thus the converse to
Proposition~\ref{prop-I-smash-I} holds.  Summarizing, we have the following.

\begin{corollary}\label{cor-smash-summary}
Fix subsets $S$ and $T$ of $\N$.
\begin{enumerate}
\item $\M{S} \Smash \I{T} = 0$ if and only if $S\mycomp \cap T$ is
infinite.
\item $\I{S} \Smash \I{T} = 0$ if and only if $S \cup T$ is infinite.
\item $\M{S} \Smash \M{T} \neq 0$ for all $S$, $T$.
\end{enumerate}
\end{corollary}

Since $k=\M{\emptyset}=\I{\emptyset}$, we can also read off its
behavior under the derived tensor product.

\begin{proof}
Part (a) is Lemma~\ref{lemma-smash}.  Part (b) is
Proposition~\ref{prop-I-smash-I} and its converse, since $S \cup T$ is
infinite if and only if either $S$ or $T$ is infinite.  Part (c)
follows from part (a), for instance: we know that $\bousfield{\M{T}}
\geq \bousfield{\M{\emptyset}} = \bousfield{\I{\emptyset}}$ for any
$T$, and part (a) tells us that $\M{S} \Smash \I{\emptyset} \neq 0$.
\end{proof}

\section{More Bousfield class calculations}\label{sec-more-bousfield}

How do the Bousfield classes of the $\M{S}$'s fit together in the
Bousfield lattice?  We know how they are ordered; can any of them be
built from the others if they are not Bousfield equivalent?  The
answer is no.

Write $\vee$ for the least upper bound operation on Bousfield
classes.  This operation is easy to understand: it is easy to verify
that $\bousfield{X} \vee \bousfield{Y} = \bousfield{X \oplus Y}$,
and more generally, that 
\[
\bigvee_{\alpha} \bousfield{X_{\alpha}} =
\bousfield{\bigoplus_{\alpha} X_{\alpha}}.
\]

\begin{lemma}\label{lemma-join}
Fix subsets $S$ and $T$ of $\N$.  
\begin{enumerate}
\item Then $\bousfield{\M{S \cup T}} \geq \bousfield{\M{S}} \vee
\bousfield{\M{T}}$.
\item If $S \not \lecomm T$ and $T \not \lecomm S$, which is to say
if $\bousfield{\M{S}}$ and $\bousfield{\M{T}}$ are not comparable in
the Bousfield lattice, then this inequality is strict:
$\bousfield{\M{S \cup T}} > \bousfield{\M{S}} \vee \bousfield{\M{T}}$.
\end{enumerate}
\end{lemma}

Of course, if for example $S \lecomm T$, then $\bousfield{\M{S}} \leq
\bousfield{\M{T}}$ and $S \cup T \comm T$, so we have equality:
$\bousfield{\M{S \cup T}} = \bousfield{\M{T}} = \bousfield{\M{S}} \vee
\bousfield{\M{T}}$.

\begin{proof}
(a) The inequality $\bousfield{\M{S \cup T}} \geq \bousfield{\M{S}}
\vee \bousfield{\M{T}}$ follows from Lemma~\ref{lemma-build-M}(b).

(b) If $S \not \lecomm T$ and $T \not \lecomm S$, then $S \cup T - S$
and $S \cup T - T$ are both infinite.  Therefore by
Lemma~\ref{lemma-smash},
\[
\M{S} \Smash \I{S \cup T} = 0 = \M{T} \Smash \I{S \cup T},
\]
while $\M{S \cup T} \Smash \I{S \cup T} \neq 0$.
Therefore the inequality from part (a) is strict.
\end{proof}

Dually, we have the following.

\begin{lemma}
Fix subsets $S$ and $T$ of $\N$.  
\begin{enumerate}
\item Then $\bousfield{\I{S \cap T}} \geq \bousfield{\I{S}} \vee
\bousfield{\I{T}}$.
\item If $S \not \lecomm T$ and $T \not \lecomm S$, which is to say
if $\bousfield{\I{S}}$ and $\bousfield{\I{T}}$ are not comparable in
the Bousfield lattice, then this inequality is strict:
$\bousfield{\I{S \cap T}} > \bousfield{\I{S}} \vee \bousfield{\I{T}}$.
\end{enumerate}
\end{lemma}

\begin{proof}
(a) See Lemma~\ref{lemma-build-I}(b).

(b) Since $S \not \lecomm T$ and $T \not
\lecomm S$, both $S - S \cap T$ and $T - S \cap T$ are infinite.
Therefore by Lemma~\ref{lemma-smash},
\[
\I{S} \Smash \M{S \cap T} = 0 = \I{T} \Smash \M{S \cap T},
\]
while $\I{S \cap T} \Smash \M{S \cap T} \neq 0$.
\end{proof}

\begin{notation}\label{notation-M}
For any subset $S$ of $\N$, let $\k{S}$ denote the trivial
$\alg{S}$-module $k$.  Given a partition $\N = A \amalg B \amalg C$,
define a module $\module{A,B,C}$ over $\Alg \cong \alg{A} \otimes
\alg{B} \otimes \alg{C}$ by
\[
\module{A,B,C} = \M{A} \otimes \k{B} \otimes \I{C}.
\]
\end{notation}

Then note, for example, that if $A$ and $A'$ are disjoint subsets of
$\N$, as a $\alg{A\cup A'} \cong \alg{A} \otimes \alg{A'}$-module, we
have $\I{A \cup A'} \cong \I{A} \otimes \I{A'}$. 
We can compute derived tensor products involving these objects with
computations like the following:
\begin{align}\label{eqn-smash-MIT}
\begin{split}
\module{A,B,C} \Smash \M{S} &= \left(\M{A} \otimes \k{B} \otimes
\I{C}\right) \Smash \M{S} \\
& \cong (\M{A} \ltensor{\alg{A}} \M{A \cap S}) \\
& \quad \otimes (\k{B} \ltensor{\alg{B}} \M{B \cap S}) \\
& \quad \otimes (\I{C} \ltensor{\alg{C}} \M{C \cap S}).
\end{split}
\end{align}
In the case of $\module{A,B,C} \Smash \module{S,T,U}$, there is a
similar decomposition, but into nine tensor factors rather than
three, yielding the following.

\begin{lemma}\label{lemma-smash-MIT}
Given two partitions $A \amalg B \amalg C$ and $S \amalg T \amalg U$
of $\N$, then
\[
\module{A,B,C} \Smash \module{S,T,U} = 0
\]
if and only if $(C \cap T) \cup (C \cap U) \cup (B \cap U)$ is infinite.
\end{lemma}

\begin{proof}
This follows from tensor product computations like
\eqref{eqn-smash-MIT}, together with
Corollary~\ref{cor-smash-summary}.
\end{proof}

The results above give us the following.

\begin{proposition}\label{prop-MIT}
Given partitions $A \amalg B \amalg C$ and $A' \amalg B' \amalg C'$ of
$\N$, then the following are equivalent:
\begin{enumerate}
\item [(1)] $(A,B,C)$ is less than or equal to $(A',B',C')$ in the
left lexicographic commensurability order -- that is, either $A
\lecomm A'$ and $A \not \comm A'$, or $A \comm A'$ and $B \lecomm B'$.
\item [(2)] $\module{A,B,C}$ may be built from $\module{A',B',C'}$.
\item [(3)] $\bousfield{\module{A,B,C}} \leq \bousfield{\module{A',B',C'}}$.
\end{enumerate}
\end{proposition}

\begin{proof}
First suppose that $A \lecomm A'$.  By Lemma~\ref{lemma-build-M}, we
can build $\M{A}$ from $\M{A'}$.  By Lemmas~\ref{lemma-build-M} and
\ref{lemma-build-I}, we can build $\I{C}$ from $\M{C \cap A'} \otimes
\k{C \cap B'} \otimes \I{C \cap C'}$.  That is, as $\alg{C} =
\alg{C \cap A'} \otimes \alg{C \cap B'} \otimes \alg{C
\cap C'}$-modules, we can write
\[
\I{C} = \I{C \cap A'} \otimes \I{C \cap B'} \otimes \I{C \cap
C'},
\]
and then we can build this module from
\[
\M{C \cap A'} \otimes \k{C \cap B'} \otimes \I{C \cap C'}.
\]

Similarly, suppose that $A \comm A'$ and $B \lecomm B'$; in this case,
$C' \lecomm C$.  Then we can build $I(C)$ from
\[
\k{C \cap B'} \otimes I(C \cap C') = \k{C \cap B'} \otimes
I(C').
\]
This finishes the proof that (1) implies (2).  (2) implies (3) in
general.

Now assume that (1) fails.  That is, assume either that $A \not
\lecomm A'$, or that $A \comm A'$ and $B \not \lecomm B'$.  If $A \not
\lecomm A'$, then $A$ is not commensurable with any subset of $A'$, so
$A$ is not commensurable with $A \cap A'$; this means that $A - A \cap
A'$ is infinite.  Lemma~\ref{lemma-smash-MIT} then tells us that 
\[
\module{A,B,C} \Smash \module{C, B, A} \neq 0, \quad 
\module{A',B',C'} \Smash \module{C, B, A} = 0.
\]
Similarly, if $A \comm B$ and $B \not \lecomm B'$, then $B-B\cap B'$
is infinite, as is $C'-C \cap C'$; indeed, 
\[
B-B \cap B' \comm C'-C \cap C'.
\]
Therefore we see again that
\[
\module{A,B,C} \Smash \module{C, B, A} \neq 0, \quad 
\module{A',B',C'} \Smash \module{C, B, A} = 0.
\]
Thus if (1) fails, then $\bousfield{\module{A,B,C}} \not \leq
\bousfield{\module{A',B',C'}}$, so (3) fails.  This finishes the
proof.
\end{proof}

We consider the subposet $\latt$ of the Bousfield lattice $\B$
consisting of the classes $\bousfield{\module{A,B,C}}$.  Then $\latt$
has a largest element $\hat{1} := \bousfield{\Alg}$ and a smallest
element $\hat{0}:=\bousfield{0}$.  We show in
Corollary~\ref{cor-minimal} that $\bousfield{\I{\N}}$ is the unique
minimum nonzero element in $\B$ (hence in $\latt$): for any nonzero
object $X$ in $\derived{\Alg}$, and in particular if $X$ happens to be
of the form $X=\module{A,B,C}$, then $\bousfield{X} \geq
\bousfield{\I{\N}} > \hat{0}$.

The poset $\latt$ has the property that, given
\[
\bousfield{\I{\N}} \leq \bousfield{\module{A,B,C}} <
\bousfield{\module{A',B',C'}} \leq \hat{1},
\]
the interval
\[
[\bousfield{\module{A,B,C}}, \bousfield{\module{A',B',C'}}]
\]
in $\latt$ contains an uncountable antichain -- see
Remark~\ref{remark-uncountable} and Corollary~\ref{cor-uncountable}.

Now define $\Latt$ to be the lattice obtained by closing $\latt$ under
arbitrary joins in $\B$, so the elements of $\Latt$ are the Bousfield
classes of direct sums of the objects $\module{A,B,C}$.  By
Theorem~\ref{thm-big}, $\Latt$ has cardinality at least
$2^{2^{\aleph_{0}}}$.

We end this section with a few questions.

\begin{question}
Is $\B = \Latt$?  That is, is every object in $\derived{\Alg}$
Bousfield equivalent to an object of the form
\[
\bigoplus_{\alpha} \module{A_{\alpha},B_{\alpha},C_{\alpha}}?
\]
\end{question}

Note that a positive answer would provide headway toward solving
\cite[Conjecture 9.1]{hovey-palmieri;bousfield}, that every localizing
subcategory in $\derived{\Alg}$ is a Bousfield class: according to
\cite[Proposition 9.2]{hovey-palmieri;bousfield}, the conjecture is
equivalent to the statement that $X$ can be built from $Y$ if and only
if $\bousfield{X} \leq \bousfield{Y}$, for all $X$ and $Y$.  Results
like Proposition~\ref{prop-MIT} seem like progress in this direction.

On the other hand, a positive answer to this question may very well be
too much to ask.  Here is a variant.

\begin{question}\label{question-bousfield}
Let $R$ be a commutative ring.  Is every object in $\derived{R}$
Bousfield equivalent to an $R$-module?
\end{question}

This is true if $R$ is noetherian, by Neeman's work
\cite{neeman;derived}.  In general, one might guess that any object
$X$ is Bousfield equivalent to the sum of its homology groups,
$\bigoplus_{i} H^{i}(X)$.


Finally, we have the following.

\begin{question}\label{question-set}
Let $R$ be a commutative ring.  Does the Bousfield lattice for
$\derived{R}$ form a set?
\end{question}

The answer is ``yes'' if $R$ is noetherian or countable: in the
noetherian case, Neeman's work \cite{neeman;derived} establishes a
bijection between the Bousfield lattice and the lattice of subsets of
$\Spec R$.  In the countable case, $\derived{R}$ is a Brown category
-- see the discussion just before Notation~\ref{notation-MIT} -- and
so Ohkawa's theorem holds, as explained in
\cite{dwyer-palmieri;ohkawa}.

\section{Objects with large tensor-nilpotence height}
\label{sec-tensor-height}

If $X$ is an object in $\derived{\Alg}$, write $X^{(n)}$ for the
$n$-fold derived tensor product of $X$ with itself.

\begin{theorem}\label{thm-tensor}
\begin{enumerate}
\item For each $n \geq 1$, there is an object $X$ so that
$X^{(n)}$ is nonzero while $X^{(n+1)}$ is zero.
\item There is an object $Y$ for which
\[
\bousfield{Y} > \bousfield{Y^{(2)}} > \bousfield{Y^{(3)}} >
\dotsb.
\]
\end{enumerate}
\end{theorem}

\begin{proof}
(The idea for this proof arose from a conversation with James
Zhang.)

(a) Fix an integer $n \geq 1$ and choose subsets $S_{i} \subseteq \N$, $1
\leq i \leq n$, so that each $S_{i}$ is infinite and the $S_{i}$'s
are pairwise disjoint.  Define a $\Alg$-module $X$ by
\[
X = \bigoplus_{i} \I{S_{i}} \otimes \M{S_{i}\mycomp}
= \bigoplus_{i} \module{S_{i}\mycomp, \emptyset, S_{i}}
\]
with notation as in \ref{notation-M}.  For disjoint subsets $S, T
\subseteq \N$, we have
\[
\left( \I{S} \otimes \M{S\mycomp} \right) \Smash \left( \I{T}
\otimes \M{T\mycomp} \right) = \I{S \cup T} \otimes \M{(S \cup
T)\mycomp},
\]
and if $S$ is infinite, then 
\[
\left( \I{S} \otimes \M{S\mycomp} \right) \Smash \left( \I{S}
\otimes \M{S\mycomp} \right) = 0.
\]
Therefore we have
\begin{align*}
X \Smash X &= \bigoplus_{1 \leq i<j \leq n} \I{S_{i} \cup S_{j}} \otimes
\M{(S_{i} \cup S_{j})\mycomp}, \\
X^{(m)} &= \bigoplus_{1 \leq i_{1} < \dotsb < i_{m} \leq n}
\I{S_{i_{1}} \cup \dotsb \cup S_{i_{m}}} \otimes \M{(S_{i_{1}} 
\cup \dotsb \cup S_{i_{m}})\mycomp}, \\
X^{(n)} &= \I{S_{1} \cup \dotsb \cup S_{n}} \neq 0, \\
X^{(n+1)} &= 0.
\end{align*}
This completes the proof of part (a).

The proof of part (b) is similar.  As in the proof of
Corollary~\ref{cor-uncountable}, for each prime number $p$, we let
$P_{p}$ be the set of powers of $p$.  Then define $Y$ by
\[
Y = \bigoplus_{p} \I{P_{p}} \otimes \M{P_{p}\mycomp}.
\]
Then, computing as in part (a), we have
\[
Y^{(m)} = \bigoplus_{p_{1} < \dotsb < p_{m}} \I{P_{p_{1}} \cup \dotsb
\cup P_{p_{m}}} \otimes \M{(P_{p_{1}} \cup \dotsb
\cup P_{p_{m}})\mycomp}.
\]
List the prime numbers as $p(1)$, $p(2)$, $p(3)$, \dots, and let 
\[
U= P_{p(m)} \cup P_{p(m+1)} \cup P_{p(m+2)} \cup \dotsb.
\]
Then $\M{U\mycomp} \Smash Y^{(m)}$ is zero by
Lemma~\ref{lemma-smash}, because $U\mycomp$ has an infinite
intersection with any $m$-fold union of the $P_{p}$'s.  On the other
hand, $\M{U\mycomp} \Smash Y^{(m-1)}$ is nonzero, because $U\mycomp
\cap (P_{p(1)} \cup \dotsb \cup P_{p(m-1)})$ is empty.
\end{proof}

As in the usual stable homotopy category
\cite{bousfield;boolean,hovey-palmieri;bousfield}, we define the
distributive lattice $\DL$ to be the sub-poset of the Bousfield
lattice $\B$ consisting of classes $\bousfield{X}$ with $\bousfield{X}
= \bousfield{X \Smash X}$; this does not contain $\bousfield{\I{S}}$
if $S$ is infinite, but it does contain the classes
$\bousfield{\M{S}}$, because each $\alg{S}$ is a ring object.

The methods in \cite[Section 3]{hovey-palmieri;bousfield} apply here
to give an epimorphism of lattices $r: \B \rightarrow \DL$, right
adjoint to the inclusion $\DL \rightarrow \B$; furthermore, the map
$r$ factors through an epimorphism
\[
r': \B / J \rightarrow \DL,
\]
where $J$ is the ideal in $\B$ of Bousfield classes less than
$\bousfield{k}$.  The ``retract conjecture'' \cite[Conjecture
3.12]{hovey-palmieri;bousfield} says that $r'$ is an isomorphism.

\begin{corollary}\label{cor-retract}
The retract conjecture fails in $\derived{\Alg}$.
\end{corollary}

\begin{proof}
According to 
\cite[Proposition 3.13(e)]{hovey-palmieri;bousfield}, the
retract conjecture implies that every object $E$ in
$\derived{\Alg}$ satisfies
\[
\bousfield{E \Smash E} = \bousfield{E \Smash E \Smash E}.
\]
The theorem says that this is not true.
\end{proof}

\section[Minimality of $\I{\N}$]
{Minimality of {\protect$\bousfield{\I{\N}}$}}\label{sec-minimal}

In this section, we prove the following theorem.

\begin{theorem}\label{thm-minimal}
Fix objects $E$ and $X$ in $\derived{\Alg}$.  If $E \neq 0$ and
$X \Smash E=0$, then $\RHom_{\Alg}(X,\Alg)=0$ and
$\Hom_{\derived{\Alg}}(X,\Alg)=0$.
\end{theorem}

A stable homotopy theorist would translate this as, ``The object
$\Alg$ in $\derived{\Alg}$ is local with respect to every
nonzero homology theory.''  Thus the theorem says that every nonzero
object has a finite local, namely $\Alg$.  This verifies the
``Dichotomy conjecture'' from \cite[Conjecture
7.5]{hovey-palmieri;bousfield} in the category $\derived{\Alg}$.
(Since no object can have both a finite local and a finite acyclic,
this also implies that no nonzero object has a finite acyclic; this
also follows from Corollary~\ref{cor-thick}(b) below.)

Theorem~\ref{thm-minimal} has a number of useful consequences. 
Taking the contrapositive gives the following.

\begin{corollary}
If $\RHom_{\Alg}(X, \Alg) \neq 0$, then $X$ is
Bousfield equivalent to $\Alg$.
\end{corollary}

\begin{corollary}\label{cor-minimal}
$\I{\N}$ is the minimum nonzero Bousfield class: that is, for every
nonzero object $E$, we have $\bousfield{E} \geq \bousfield{\I{\N}}$.
\end{corollary}

\begin{proof}
Fix a nonzero object $E$ in $\derived{\Alg}$ and suppose that $Y
\Smash E=0$.  Then $\RHom_{\Alg} (Y, \Alg)=0$, so by
Corollary~\ref{cor-bcdual}, $Y \Smash \I{\N}=0$.  That is, whenever $Y
\Smash E=0$, we have $Y \Smash \I{\N}=0$.  Thus $\bousfield{E} \geq
\bousfield{\I{\N}}$.
\end{proof}

As in the usual stable homotopy category, we define the Boolean
algebra $\BA$ to be the set of \emph{complemented} Bousfield classes:
those classes $\bousfield{X}$ for which there is a $Y$ with
\[
\bousfield{X \Smash Y} = 0, \quad \quad 
\bousfield{X \oplus Y} = \Alg. 
\]
In any stable homotopy category, we have the inclusion $\BA \subseteq
\DL$; see \cite{bousfield;boolean} and \cite{hovey-palmieri;bousfield}
for more about $\BA$, $\DL$, and $\B$.

\begin{corollary}\label{cor-bool}
In the Bousfield lattice for the category $\derived{\Alg}$,
$\BA$ is trivial: the only complemented classes are
$\bousfield{0}$ and $\bousfield{\Alg}$.
\end{corollary}

\begin{proof}
Fix $\bousfield{X}$ not equal to $\bousfield{0}$ or
$\bousfield{\Alg}$.  Suppose that $Y \neq 0$ and $X \Smash Y=0$.
Then by Corollary~\ref{cor-minimal}, $\bousfield{X} \geq
\bousfield{\I{\N}}$ and $\bousfield{Y} \geq \bousfield{\I{\N}}$.  Thus
$(X \oplus Y) \Smash \I{\N} = 0$, so there is no object $Y$ so that $X
\Smash Y=0$ and $\bousfield{X \oplus Y} = \bousfield{\Alg}$.
\end{proof}

The \emph{telescope conjecture} in a general stable homotopy category
says that every smashing localization is a finite localization -- see
\cite[3.3.8]{axiomatic}.

\begin{corollary}
The telescope conjecture holds in $\derived{\Alg}$.
\end{corollary}

\begin{proof}
Suppose that $\bousfield{E}$ is a nonzero smashing Bousfield class.
From \cite[1.31]{ravenel;localization} or
\cite[4.5]{hovey-palmieri;bousfield}, we have that $\bousfield{E}$
must be complemented, and so by Corollary~\ref{cor-bool}, we have
$\bousfield{E}=\bousfield{\Alg}$.  Localization with respect to
$\Alg$ is the identity functor, which is trivially a finite
localization.
\end{proof}

We need a few lemmas before we prove the theorem.

Fix a ring $R$ and a subring $S$; then for any object $Y$ in
$\derived{R}$, write $\restrict{Y}{S}$ for the restriction of $Y$ to
$S$.

\begin{lemma}\label{lemma-bimodule}
Let $S$ be a graded commutative ring, and let $R=S[x]/(x^{n})$.
Assume that $R$ is graded commutative; that is, if $n>2$, assume that
$\deg x$ is even.  For any object $X$ in $\derived{R}$, one can build
$R \ltensor{S} \restrict{X}{S}$ from $X$.
\end{lemma}

\begin{proof}
The proof goes like this: first we claim that one can build $R
\tensor{S} R = R \tensor{S} \restrict{R}{S}$ from $R$.  Given this, we
apply $- \ltensor{R} X$ to deduce that one can build $R \ltensor{S}
\restrict{X}{S}$ from $X$.  There is an important technical point,
though: while it is clear that one can build $R \tensor{S} R$ from $R$
as, say, a right $R$-module (since one can build anything in the
derived category of right $R$-modules from $R$), then after tensoring
with $X$, we would only know that we could build $R \ltensor{S} X$
from $X$ as a left $S$-module, not as a left $R$-module.

So the lemma is equivalent to this claim: one can build $R \tensor{S}
R$ from $R$ \emph{as an $(R,R)$-bimodule}.  Of course,
$(R,R)$-bimodules are the same as modules over $R \tensor{S}
R^{\text{op}} \cong R \tensor{S} R \cong S[x,y]/(x^{n},y^{n})$.  So we
make $S[x]/(x^{n})$ into a module over $S[x,y]/(x^{n},y^{n})$ by
having both $x$ and $y$ act as multiplication by $x$; then we want to
show that we can build $S[x,y]/(x^{n}, y^{n})$ from $S[x]/(x^{n})$.
We do this in two steps.

First, we build $S$ from $S[x]/(x^{n})$ in
$\derived{S[x,y]/(x^{n},y^{n})}$.  We do this by taking a free
resolution of $S$ as an $S[x]/(x^{n})$-module, treating the result as
a chain complex of (non-free) modules over $S[x,y]/(x^{n},y^{n})$, and
then taking successive truncations as in the proof of
Lemma~\ref{lemma-build-M}.  That is, $S$ is the colimit of the chain
complexes 
\[
0 \rightarrow P_{m} \rightarrow \dotsb \rightarrow P_{0} \rightarrow 0,
\]
where each $P_{i}$ is a shifted copy of $S[x]/(x^{n})$.  Each such
chain complex can be built from $S[x]/(x^{n})$, and hence so can $S$.

Second, using the fact that $S[x,y]/(x^{n},y^{n})$ is free of finite
rank as an $S$-module, we can build it from $S$.  Since we can build
$S$ from $S[x]/(x^{n})$, this completes the proof.
\end{proof}

\begin{lemma}\label{lemma-restrict}
Let $T$ be a subset of $\N$ with $\N -T$ finite.  
Fix an object $Y$ in $\derived{\Alg}$.  If the free module $\alg{T}$
is a summand of $\restrict{Y}{\alg{T}}$, then $\Alg$ may be built from
$Y$, so $\bousfield{Y} = \bousfield{\Alg}$.
\end{lemma}

\begin{proof}
First assume that $\Alg$ may be built from $Y$; then $\bousfield{Y} \geq
\bousfield{\Alg}$; on the other hand, $\bousfield{Y} \leq
\bousfield{\Alg}$ for any $Y$, so we see that
$\bousfield{Y}=\bousfield{\Alg}$.

So we need to verify that $\Alg$ may be built from $Y$.  Suppose
that $\N - T = \{x_{i_{1}}, \dotsc, x_{i_{n}}\}$.  For each $j$, $1
\leq j \leq n$, let
\[
\Alg_{j} = \alg{T \cup \{x_{i_{1}},\dotsc, x_{i_{j}} \} },
\]
and let $\Alg_{0} = \alg{T}$.  Note that $\Alg_{n} = \Alg$.  Then for
each $j \leq n$, there is an algebra isomorphism
\[
\Alg_{j} \cong \Alg_{j-1}[x_{j}]/(x_{j}^{n_{j}}),
\]
so we may apply the previous lemma to conclude that we may build
$\Alg_{j} \ltensor{\Alg_{j-1}} \restrict{Y}{\Alg_{j-1}}$ from
$\restrict{Y}{\Alg_{j}}$ for each $j$.  Since $\Alg_{0}$ is a summand
of $\restrict{Y}{\Alg_{0}}$, we may build $\Alg_{1} \ltensor{\Alg_{0}}
\Alg_{0} \cong \Alg_{1}$ from $\restrict{Y}{\Alg_{1}}$.  Inductively,
we see that we may build $\Alg_{j}$ from $\restrict{Y}{\Alg_{j}}$ for
each $j$, and in particular when $j=n$, we may build $\Alg_{n}=\Alg$
from $Y$.
\end{proof}

\begin{proof}[Proof of Theorem \ref{thm-minimal}]
Suppose that $X$ is an object of $\derived{\Alg}$, and suppose that
$\Hom_{\derived{\Alg}}(X, \Alg) \neq 0$.  We want to show
that if $X \Smash E = 0$, then $E=0$; that is, we want to show that
$X$ is Bousfield equivalent to the sphere object $\Alg$.

So suppose that $f: X \rightarrow \Alg$ is a nonzero map.  Then one
can check that it induces a nonzero map on homology: $H(f) : H^{0}(X)
\rightarrow \Alg$ is nonzero.  This is a map of $\Alg$-modules.  If
$H(f)(y) \neq 0$, then the element $y \in H^{0}(X)$ supports
free actions by all but finitely many of the $x_{i}$'s, because
this is true for its target in $\Alg$.  Let $T$ be the set
\[
T = \{i: x_{i}^{n_{i}-1} y \neq 0 \}.
\]
Then $T$ is cofinite in $\N$, and in the homology of the restriction
of $X$ to $\alg{T}$, $y$ will generate a free module of rank 1.
Therefore this free module will split off of $\restrict{X}{\alg{T}}$.
By Lemma~\ref{lemma-restrict}, this means that
$\bousfield{X}=\bousfield{\Alg}$, as desired.
\end{proof}

\section{Nilpotence and small objects}\label{sec-nilp}

In this section, we discuss versions of some stable homotopy results
in the category $\derived{\Alg}$.  In particular, we show that the
object $k$ detects nilpotence, in the language of \cite{hopkins-smith,
hopkins;global}.  We use this to classify the thick subcategories of
small objects, and also to prove a version of the periodicity theorem
\cite{hopkins-smith,hopkins;global}.

Note that for any object $X$ in the derived category
$\derived{\Alg}$, its homology is
$\Hom_{\derived{\Alg}}^{*}(\Alg, X)$.  This is not the homology
theory determined by $k$; indeed, if $M$ is a $\Alg$-module, then
\[
k_{*}M := \Hom_{\derived{\Alg}}^{*}(\Alg, k \Smash M) \cong
H_{*}(k \Smash M) \cong \Tor^{\Alg}_{*}(k,M).
\]

In order to state the nilpotence theorem, we need to recall a few
definitions.

\begin{definition}
\begin{enumerate}
\item 
First, $E$ is a \emph{ring object} in $\derived{\Alg}$ if there is
a ``multiplication'' map $\mu : E \Smash E \rightarrow E$ and a
``unit'' map $\eta : \Alg \rightarrow E$ in $\derived{\Alg}$ so
that $\mu \circ (\eta \otimes 1) = 1 = \mu \circ (1 \otimes \eta)$.
If $E$ is a ring object, then $\Hom_{\derived{\Alg}}^{*}(\Alg ,
E)$ is a (graded) ring: given two maps $f, g: \Alg \rightarrow E$,
their product is the composite
\[
\Alg \xrightarrow{\cong} \Alg \Smash \Alg \xrightarrow{f
\otimes g} E \Smash E \xrightarrow{\mu} E.
\]
Similarly, since $k$ is a ring object, so is $k \Smash E$, and so
$k_{*}E = \Hom_{\derived{\Alg}}^{*}(\Alg, k \Smash E)$ is a
graded ring.
\item 
An object $F$ in $\derived{\Alg}$ is \emph{small} if the natural
map 
\[
\bigoplus_{\alpha} \Hom_{\derived{\Alg}}(F, X_{\alpha}) \rightarrow
\Hom_{\derived{\Alg}} (F, \bigoplus_{\alpha} X_{\alpha})
\]
is an isomorphism for any set of objects $\{X_{\alpha}\}$.  
\item 
A map $g: F \rightarrow F$ is \emph{nilpotent} if some
composite $g \circ \dotsb \circ g$ is zero.
A map $f: F \rightarrow X$ is \emph{tensor-nilpotent} if for
some $n\geq 1$, the map 
\[
f^{\otimes n} : \underbrace{F \Smash \dotsb \Smash F}_{n} \rightarrow
\underbrace{X \Smash \dotsb \Smash X}_{n}
\]
is zero.
\end{enumerate}
\end{definition}

\begin{theorem}[Nilpotence theorem]
\label{thm-nilp}
The field object $k=\alg{\emptyset}$ detects nilpotence:
\begin{enumerate}
\item For any ring object $E$ and $\alpha: \Alg \rightarrow E$,
$\alpha$ is nilpotent if and only if $k_{*}\alpha$ is nilpotent.
\item For any small object $F$ and self-map $g:F \rightarrow F$, $g$
is nilpotent if and only if $k_{*}g$ is nilpotent.
\item For any small object $F$, any object $X$, and any map $f: F
\rightarrow X$, the map $f$ is tensor-nilpotent if and only if
$k_{*}f$ is zero.
\end{enumerate}
\end{theorem}

\begin{proof}
Since $\Alg$ is the unit of the derived tensor product in
$\derived{\Alg}$, it plays the role of the sphere object.
Therefore, as in the proof of \cite[Theorem 3]{hopkins-smith}, the
proof boils down to this situation: $f: \Alg \rightarrow X$ is a
map, and $T$ is the following telescope:
\[
T = \varinjlim \left(\Alg \xrightarrow{f} X \xrightarrow{f \otimes
1_{X}} X \Smash X \xrightarrow{f \otimes 1_{X} \otimes 1_{X}} X \Smash X \Smash X
\xrightarrow{} \dotsb\right).
\]
We assume that $k_{*}f=0$, and we want to show that the $m$th
derived tensor power $f^{(m)}$ of $f$ is zero for $m \gg 0$, or
equivalently, that $T=0$.  From Lemma~\ref{lemma-build-M}, we see that
for any $n \geq 1$,
\[
\bousfield{\Alg} = \bousfield{\M{n, n+1, \dotsc}}.
\]
Since $k = \varinjlim_{n} \M{n, n+1, \dotsc}$, the map 
\[
\Alg \rightarrow \varinjlim \M{n, n+1, \dotsc} \Smash T
\]
is null.  Thus since $\Alg$ is small in the category
$\derived{\Alg}$, we see that for some $n \geq 1$,
\[
\Alg \rightarrow \M{n, n+1, \dotsc} \Smash T
\]
is null, which means that 
\[
\M{n, n+1, \dotsc} \Smash T = 0,
\]
and hence $T = \Alg \Smash T = 0$, as desired.
\end{proof}

It is a standard result that the subcategory of small objects in the
derived category of a commutative ring $R$ is precisely the thick
subcategory (Definition~\ref{defn-thick}) generated by $R$.  Also,
whenever one has field objects which detect nilpotence, they determine
the thick subcategories of the category of small objects -- see
\cite[Corollary 5.2.3]{axiomatic}.  Thus since $k$ is a field object,
Theorem~\ref{thm-nilp} gives the following.

\begin{corollary}[Thick subcategory theorem]\label{cor-thick}
Consider the category $\cat{F}=\thick(\Alg)$ of small objects of
$\derived{\Alg}$.
\begin{enumerate}
\item Every thick subcategory of $\cat{F}$ is trivial: the only thick
subcategories are $\{0\}$ and $\cat{F}$.
\item Every nonzero small object in $\derived{\Alg}$ is Bousfield
equivalent to the sphere object $\Alg$.
\end{enumerate}
\end{corollary}

Part (a) also follows from a result of Thomason \cite[Theorem
3.15]{thomason;triangulated}.

\begin{proposition}[Nishida's theorem, periodicity theorem]
\label{prop-periodicity}
For any small object $F$ in $\derived{\Alg}$, any self-map $g:F
\rightarrow F$ of nonzero degree is nilpotent.
\end{proposition}

\begin{proof}
Fix a map $g: \Sigma^{a,b} F \rightarrow F$ with $(a,b) \neq (0,0)$.
Since $k_{*}\Alg$ consists of a single copy of $k$ in degree
$(0,0)$, we see that for any object $F$ in $\thick (\Alg)$,
$k_{*}F$ is finite-dimensional as a vector space over $k$.  Thus for
degree reasons, $k_{*}g$ must be nilpotent, so Theorem~\ref{thm-nilp}
says that $g$ is nilpotent.
\end{proof}





\def\cftil#1{\ifmmode\setbox7\hbox{$\accent"5E#1$}\else
  \setbox7\hbox{\accent"5E#1}\penalty 10000\relax\fi\raise 1\ht7
  \hbox{\lower1.15ex\hbox to 1\wd7{\hss\accent"7E\hss}}\penalty 10000
  \hskip-1\wd7\penalty 10000\box7} \def\cprime{$'$}
\providecommand{\bysame}{\leavevmode\hbox to3em{\hrulefill}\thinspace}
\providecommand{\MR}{\relax\ifhmode\unskip\space\fi MR }
\providecommand{\MRhref}[2]{%
  \href{http://www.ams.org/mathscinet-getitem?mr=#1}{#2}
}
\providecommand{\href}[2]{#2}


\begin{thebibliography}{HPS97}

\bibitem[BC76]{brown-comenetz}
E.~H. Brown and M.~Comenetz, \emph{Pontrjagin duality for generalized homology
  and cohomology theories}, Amer. J. Math. \textbf{98} (1976), no.~1, 1--27.
  \MR{53 \#9196}

\bibitem[Bou79]{bousfield;boolean}
A.~K. Bousfield, \emph{The {B}oolean algebra of spectra}, Comment. Math. Helv.
  \textbf{54} (1979), no.~3, 368--377. \MR{81a:55015}

\bibitem[CE56]{cartan-eilenberg}
H.~Cartan and S.~Eilenberg, \emph{Homological algebra}, Princeton University
  Press, Princeton, N. J., 1956.

\bibitem[DP01]{dwyer-palmieri;ohkawa}
W.~G. Dwyer and J.~H. Palmieri, \emph{Ohkawa's theorem: there is a set of
  {B}ousfield classes}, Proc. Amer. Math. Soc. \textbf{129} (2001), no.~3,
  881--886. \MR{1 712 921}

\bibitem[Hop87]{hopkins;global}
M.~J. Hopkins, \emph{Global methods in homotopy theory}, Homotopy theory
  (Durham, 1985) (J.~D.~S. Jones and E.~Rees, eds.), Cambridge Univ. Press,
  Cambridge-New York, 1987, LMS Lecture Note Series 117, pp.~73--96.

\bibitem[HP99]{hovey-palmieri;bousfield}
M.~Hovey and J.~H. Palmieri, \emph{The structure of the {B}ousfield lattice},
  Homotopy Invariant Algebraic Structures (J.-P. Meyer, J.~Morava, and W.~S.
  Wilson, eds.), Contemp. Math., vol. 239, Amer. Math. Soc., Providence, RI,
  1999, pp.~175--196.

\bibitem[HPS97]{axiomatic}
M.~Hovey, J.~H. Palmieri, and N.~P. Strickland, \emph{Axiomatic stable homotopy
  theory}, Mem. Amer. Math. Soc. \textbf{128} (1997), no.~610, x+114.
  \MR{98a:55017}

\bibitem[HS98]{hopkins-smith}
M.~J. Hopkins and J.~H. Smith, \emph{Nilpotence and stable homotopy theory
  {I}{I}}, Ann. of Math. (2) \textbf{148} (1998), no.~1, 1--49.

\bibitem[Mar83]{margolis}
H.~R. Margolis, \emph{Spectra and the {S}teenrod algebra}, North-Holland
  Publishing Co., Amsterdam-New York, 1983, Modules over the Steenrod algebra
  and the stable homotopy category.

\bibitem[Nee92]{neeman;derived}
A.~Neeman, \emph{The chromatic tower for {D}({R})}, Topology \textbf{31}
  (1992), no.~3, 519--532.

\bibitem[Nee00]{neeman;oddball}
A.~Neeman, \emph{Oddball {B}ousfield classes}, Topology \textbf{39} (2000),
  no.~5, 931--935. \MR{MR1763956 (2001c:18007)}

\bibitem[Rav84]{ravenel;localization}
D.~C. Ravenel, \emph{Localization with respect to certain periodic homology
  theories}, Amer. J. Math. \textbf{106} (1984), no.~2, 351--414.

\bibitem[Tho97]{thomason;triangulated}
R.~W. Thomason, \emph{The classification of triangulated subcategories},
  Compositio Math. \textbf{105} (1997), no.~1, 1--27.

\end{thebibliography}
\end{document}